\documentclass[11pt]{article}
\usepackage[totalwidth=14.0cm,totalheight=20.0cm]{geometry}
\usepackage{latexsym,amsthm,amsmath,amssymb,url}
\usepackage[ruled, linesnumbered]{algorithm2e}

\usepackage{tikz,authblk}
\usepackage{comment}
\usetikzlibrary{decorations.pathreplacing}

\newtheorem{theorem}{Theorem}
\newtheorem{corollary}[theorem]{Corollary}
\newtheorem{lemma}[theorem]{Lemma}

\newtheorem{proposition}[theorem]{Proposition}
\newtheorem{problem}[theorem]{Open Problem}
\newtheorem{conjecture}{Conjecture}

\newcommand{\YZ}[1]{{#1}}

\newcommand{\fas}{{\rm fas}}
\newcommand{\acs}{{\rm acs}}

\newcommand{\AY}[1]{{#1}}
\newcommand{\LH}[1]{{#1}}

\newcommand{\GG}[1]{{#1}}
\newcommand{\2}{\vspace{0.2cm}}

\title{Upper bounds on minimum size of feedback arc set of directed multigraphs with bounded degree}

\author{G. Gutin, H. Lei, A. Yeo, Y. Zhou}
\author{ Gregory Gutin\thanks{Department of Computer Science. Royal Holloway University of London, UK, and School of Mathematical Sciences and LPMC, Nankai University, Tianjin 300071, China.  {\tt g.gutin@rhul.ac.uk}} \hspace{2mm}  Hui Lei\thanks{ School of Statistics and Data Science, LPMC and KLMDASR, Nankai University, Tianjin 300071, China. Partially supported by the NSFC grant (No. 12371351). {\tt hlei@nankai.edu.cn}} \hspace{2mm}  Anders Yeo\thanks {Department of Mathematics and Computer Science, University of Southern Denmark, Denmark, and Department of Mathematics and Applied Mathematics, University of Johannesburg, South Africa. {\tt yeo@imada.sdu.dk}} \hspace{2mm}  Yacong Zhou\thanks{Department of Computer Science. Royal Holloway University of London, UK. {\tt Yacong.Zhou.2021@live.rhul.ac.uk}} }

\begin{document}
	
	\maketitle
	
	\begin{abstract}
		An oriented multigraph is a directed multigraph without directed 2-cycles.	Let ${\rm fas}(D)$ denote the minimum size of a feedback arc set in an oriented multigraph $D$.
		The degree of a vertex is the sum of its out- and in-degrees. In several papers, upper bounds for ${\rm fas}(D)$ were obtained for oriented multigraphs $D$ with maximum degree upper-bounded by a constant. Hanauer (2017) conjectured that ${\rm fas}(D)\le 2.5n/3$ for every oriented multigraph $D$ with $n$ vertices and maximum degree at most 5.
		We prove a strengthening of the conjecture: ${\rm fas}(D)\le m/3$ holds for every oriented multigraph $D$ with $m$ arcs and maximum degree at most 5. This bound is tight and improves a bound of Berger and Shor (1990,1997).
		It would be interesting to determine $c$ such that ${\rm fas}(D)\le cn$  for every oriented multigraph $D$ with $n$ vertices and maximum degree at most 5 such that the bound is tight. We show that
		\AY{$\frac{5}{7}\le c  \le \frac{24}{29} < \frac{2.5}{3}$}.
	\end{abstract}
	
	\section{Introduction}
	\label{sec:intro}
	An {\em oriented multigraph} $D$ is a directed multigraph without loops and directed cycles of length 2. An {\em oriented graph} is an oriented multigraph with no multiple arcs. If an oriented multigraph $D$ is clear from the context, we denote its number of vertices by $n$ and number of arcs by $m$. The \LH{{\em degree} $d_{D}(v)$} of a vertex $v$ in $D$ is the number of arcs of $D$ incident to $v$ (all arcs where $v$ is the tail or head are counted). \LH{We say that $D$ is {\em degree-$k$} if every vertex in $D$ has degree $k$.} The maximum degree of $D$ is denoted by $\Delta(D)$ or just $\Delta$ if $D$ is clear from the context. A set $F$ of arcs of $D$  is a {\em feedback arc set}  if $D-F$ has no directed cycle.
	We mainly follow terminology and notation of \cite{BJG}. However, for convenience of the reader we introduce most required terminology and notation in this section.
	
	The {\sc Minimum Feedback Arc Set} problem is a well-known NP-hard problem (it is NP-hard even on tournaments \cite{Alon2006,CharbitTY2007}) with numerous applications, see e.g. \cite{Alon2002,ELS1993,LS1991}. The problem is as follows: given a directed multigraph (arc-weighted digraph $D$, resp.) find a feedback arc set $F$ of $D$ with minimum number of arcs (of minimum weight, resp.), denoted by $\fas(D)$. The problem of finding  $\fas(D)$ is complementary to the problem of finding $\acs(D)$, the maximum number (weight, resp.) of arcs in an acyclic subdigraph of $D$, i.e.,
	$\fas(D)+\acs(D)=|A(D)|$	($\fas(D)+\acs(D)=w(D)$, the weight of $A(D)$, resp.). Clearly, exactly one arc of a directed cycle of length 2 is in any feedback arc set. Thus, while studying {\sc Minimum Feedback Arc Set}, it suffices to restrict ourselves to oriented multigraphs or arc-weighted oriented graphs. In this paper, we restrict ourselves to oriented multigraphs.
	
	Studying heuristics for $\acs(D)$ on oriented graphs with maximum degree $\Delta$, Berger and Shor \cite{BS1990,BS1997} proved that $\acs(D)\ge (\frac{1}{2}+\Omega(\frac{1}{\sqrt{\Delta}}))m.$ Using \GG{the probabilistic method}, Alon \cite{Alon2002} improved the bound to $\acs(D)\ge (\frac{1}{2}+\frac{1}{16\sqrt{\Delta}})m$ for oriented multigraphs with maximum degree $\Delta$ (in fact, Alon proved such a bound for arc-weighted oriented graphs). Thus, for an oriented multigraph $D$ with maximum degree $\Delta$, $\fas(D)\le (\frac{1}{2}-\frac{1}{16\sqrt{\Delta}})m$. Jung \cite{Jung1970} and Spencer \cite{Spencer1971,Spencer1980} showed that if $t(n)$ is the maximum of $\fas(T)$ for a tournament $T$ with $n$ vertices, then
	$t(n)=(\frac{1}{2}-\Theta(\frac{1}{\sqrt{n}})){n \choose 2}$.
	It follows from the result of Jung and Spencer that the bounds of Alon, and Berger and Shor are asymptotically tight subject to a positive constant $b$ in $\frac{b}{\sqrt{\Delta}}$.
	
	However, for small values of $\Delta$, the above bounds are far from tight. Berger and Shor \cite{BS1990,BS1997} showed that for an oriented graph $D$ with $\Delta \le $ 2 and $\Delta \le $ 3, $\fas(D)\le m/3$, and for $\Delta\le $ 4 and $\Delta\le $ 5, $\fas(D)\le 11m/30.$ 
	
	Hanauer et al. \cite{HBA2013} and Hanauer \cite{Hanauer2017} improved the above bounds for  $\Delta\leq 3$ and $\Delta\leq 4$ as follows.
	
	\begin{theorem}\label{thm:md34}\cite{Hanauer2017}
		(i) If $D$ is an oriented multigraph with $\Delta\leq 3$, then $\fas(D)\leq n/3$.
		(ii) If $D$ is an oriented multigraph with $\Delta\leq 4$, then $\fas(D)\leq m/3$.
		Both bounds are tight. Furthermore, the bound of (ii) is tight for degree-$4$ oriented multigraphs.
	\end{theorem}
	
	Hanauer  \cite{Hanauer2017} posed the following:
	
	\begin{conjecture}
		If $D$ \LH{is} an oriented multigraph with $\Delta\le 5$, then
		$\fas(D)\le 2.5n/3$.
	\end{conjecture}
	
	Note that if $\Delta\le 5$ then $m\le 2.5n$. The following main result of this paper proves the above conjecture in a stronger form. It also improves the bound of Berger and Shor \cite{BS1990} for $\Delta\le $ 4 and $\Delta\le $ 5.
	
	\begin{theorem}\label{thm:main}
		If $D$ is an oriented multigraph with
		$\Delta\le$ 5, then $\fas(D)\le m/3$.
	\end{theorem}
	
	Since the bound of Theorem \ref{thm:md34} (ii) is tight, the bound in Theorem \ref{thm:main} is also tight. \YZ{Now we sketch the proof of Theorem \ref{thm:main}.
	The main idea of the proof is to use reductions whenever we can to delete at least $3k$ arcs from the input oriented multigraph $D$ for some integer $k\geq 0$ obtaining $D'$ such that $\fas(D)\le \fas(D')+k.$
We will call such reductions {\em nice reductions} (see \LH{reductions} (n0)-(n545) in Section \ref{sec:thm2}). For {\em good reductions} (see \LH{reductions} (g1)-(g5) in Section \ref{sec:thm2})
we can do a bit better: $|A(D')|\le |A(D)|-(3k+1)$ and $\fas(D)\le \fas(D')+k.$ Let $D'$ be an oriented multigraph obtained after performing all possible reductions starting from $D$. If $D'$ is empty or $\Delta(D')\le 4$ then we are done due to Theorem \ref{thm:md34}.  If $\Delta(D')=5$, we show that there is a vertex $x\in V(D')$ of out-degree two and in-degree three such that $\fas(D')\le \fas(D'')+2,$  where $D''$ is obtained from $D'$ by deleting the arcs incident with $x.$ Moreover, on $D''$, one can do a sequence of nice reductions followed by a good reduction. Then, the good reduction will compensate for the starting non-nice reduction. Indeed, assume that on $D''$, we \LH{perform} a sequence of nice reductions that \LH{delete} at least $3r$ arcs and one good reduction that \LH{deletes} at least $3k+1$ arcs. Then, we \LH{delete} at least $5+3r+3k+1=3(k+r+2)$ arcs from $D'$,
but $\fas(D')$ \LH{decreases} by at most $2+r+k$. Thus, we are done as this whole reduction sequence can be seen as one nice reduction.}
	
	\2
	
	Hanauer et al. \cite{HBA2013}  posed the following:
	
	\begin{conjecture}\label{conj2}
		If $D$ \LH{is} a strongly connected oriented graph with $\Delta\le 5$, then
		$\fas(D)\le 2n/3$.
	\end{conjecture}
	
	We obtain a counterexample to this conjecture, see 	Corollary \ref{rem} in Section \ref{sec:examples}.
	
	\2
	
		\LH{Additionally}, we will show the following theorem for degree-$5$ oriented multigraphs.
	
	\begin{theorem}\label{thm:5-rg}
		If $D$ is a degree-$5$ oriented multigraph, then $\fas (D)\leq \YZ{24n/29}$.
	\end{theorem}

	\subsection{Coefficients in Upper Bounds}\label{sec:coef}
	
	Let $c_{\leq k}'$ ($c_{\leq k}''$, respectively) be the infimum of all reals such that $\fas(D)\leq c_{\leq k}' \cdot m$ ($\fas(D) \leq c_{\leq k}'' \cdot n$, respectively) holds for oriented multigraphs $D$ with $\Delta\leq k$. And let $c_{k}'$ ($c_{k}''$, respectively) be the infimum of all reals such that $\fas(D)\leq c_{k}' \cdot m$ ($\fas(D) \leq c_{k}'' \cdot n$, respectively) holds for all degree-$k$ oriented multigraphs $D$.
\YZ{The following proposition asserts some relations between the coefficients. 	
	
\begin{proposition}\label{relations}
For every integer $k\geq 2$,
(i) $c''_k=c''_{\leq k}$, and (ii) $c''_k=c'_k \cdot (k/2).$
\end{proposition}
\begin{proof}
(ii) is a trivial observation. To prove (i), observe first that $c''_{\leq k}\geq c''_{k}$. Suppose that $\epsilon=c''_{\leq k}-c''_{k}>0$. Let $D$ be an oriented multigraph with $n$ vertices, $\Delta\leq k$ and $\fas(D)\geq (c''_{\leq k}-\frac{\epsilon}{2})\cdot n$. We can construct a degree-$k$ oriented multigraph $D'$ with $m'$ arcs and $n'$ vertices by taking two copies of $D$ and adding $k-d_D(u)$ arcs with the same direction between two vertices corresponding to every $u\in V(D)$. However, $c''_k\cdot n' \geq\fas(D')\geq (c''_{\leq k}-\frac{\epsilon}{2})\cdot n'=(c''_{k}+\frac{\epsilon}{2})\cdot n'$, a contradiction.
\end{proof}
}
	
	
Due to Proposition \ref{relations}, to determine these four coefficients for an integer $k\geq 2$, we only need to know \AY{$c'_{\leq k}$ and one of the
coefficients $c''_k$, $c''_{\leq k}$ or $c'_k$.} We will consider $c'_{\leq k}$ and $c''_{\leq k}$ to be consistent with the previous papers. Also, these two coefficients give us upper bounds for a wider range of oriented multigraphs and together they could potentially offer a better upper bound, i.e., for every oriented multigraph $D$ with $\Delta\leq k$, $\fas(D)\leq \min(c'_{\leq k}\cdot m, c''_{\leq k}\cdot n)$.
	
	\YZ{Berger  \cite{Berger1997} proved the following bound for an oriented graph $D$ with degrees $d_1,\dots ,d_n$ of vertices.
	\begin{equation}\label{eq1} \fas(D)\le \frac{m}{2} - \frac{\sqrt{6}}{40}\sum^n_{i=1}\sqrt{d_i}.\end{equation}
	By Proposition \ref{relations}(i), for every oriented graph $D$ with $\Delta\le k$, we have
	\begin{equation}\label{eq2} \fas(D)\le \left(\frac{k}{4} - \frac{\sqrt{6}}{40}\sqrt{k}\right)n.\end{equation}
	
	Combining Alon's bound and (\ref{eq2}), we have the following bound improving both Alon's and Berger's bounds for oriented graphs with $\Delta\le k$.
	\begin{equation*} \fas(D)\le \min\left(\left(\frac{k}{4} - \frac{\sqrt{6}}{40}\sqrt{k}\right)n,\left(\frac{1}{2}-\frac{1}{16\sqrt{k}}\right)m\right).
	\end{equation*}

\AY{Note that if $D$ is degree-$k$ then $\left(\frac{k}{4} - \frac{\sqrt{6}}{40}\sqrt{k}\right)n <
 \left(\frac{k}{4}-\frac{\sqrt{k}}{32}\right)n
= \left(\frac{1}{2}-\frac{1}{16\sqrt{k}}\right)m$. However, if $m$ is sufficiently smaller than $kn/2$ then $\left(\frac{1}{2}-\frac{1}{16\sqrt{k}}\right)m<\left(\frac{k}{4} - \frac{\sqrt{6}}{40}\sqrt{k}\right)n$.}

	\begin{table}[t]\label{table:1}
		\centering
		\caption{The value or range of $c'_{\leq k}$ and $c''_{\leq k}$ when $k\in [2,6]$}
		\begin{tabular}{|c||c|c|c|c|c|c|c|} \hline
			&  $k=2$  &  $k=3$  &  $k=4$  &  $k=5$  &  $k=6$           \\ \hline \hline
			{\small $c_{ \leq k}'$} &  $1/3$  &  $1/3$  &  $1/3$  &  $1/3$  &  $\geq 25/72$   \\ \hline
			{\small $c_{ \leq k}''$} &  $1/3$  &  $1/3$  &  $2/3$  &  $\YZ{\in} [\frac{5}{7}, \YZ{\frac{24}{29}}]$  & $\geq 75/72$\\ \hline
		\end{tabular}
	\end{table}
}

\2	
	Note that when $k=2$, we have $c'_{\leq 2}=c''_{\leq 2}=1/3$ as in this case, the underlying graph of $D$ is a collection of paths and cycles and therefore $\fas(D)$ is equal to the number of (directed) cycles in $D$, which is at most $n/3$. And $\fas(D)=n/3=m/3$ when $D$ consist of vertex-disjoint $3$-cycles. This fact together Theorems \ref{thm:md34}, \ref{thm:main} and \ref{thm:5-rg} and two examples in Section \ref{sec:examples}, gives Table 1, where $c''_{\leq 4}=2/3$ as $c'_{4}=1/3$ (since the 1/3 bound is tight for degree-$4$ oriented multigraphs and $c''_{\leq 4}=c''_{4}=c'_{4}\cdot 2$\AY{).}

	Note that, in the table, the lower bounds for $c_{ \leq 5}''$, $c_{ \leq 6}'$ and $c_{ \leq 6}''$ are proved for oriented graphs, which may be of interest for research restricted in oriented graphs.
	Also, note that the lower bound for $c'_{\leq 6}$ implies that $5$ is the highest integer $k$ for which $c'_{\leq k}=1/3$. Since we have been unable to determine the exact value of $c''_{\leq 5}$ and \GG{since} $c''_{\leq 5}=c''_{5}$, we \GG{pose} the following:
	
	\begin{problem} Determine $c''_{5}$.
	\end{problem}

	\subsection{Additional Notation and Terminology}
	Let $D=(V(D),A(D))$ be an oriented multigraph and $v$ a vertex in $D$. For a vertex $v\in V (D)$, we denote by $d^+_D(v)$ and $d^-_D(v)$ the {\em out-degree} and {\em in-degree} of $v$, respectively (which is the number of arcs leaving and entering $v$, respectively). Thus, $d_D(v)=d^+_D(v)+d^-_D(v).$
	By a {\em cycle} we mean a directed cycle. A cycle of length $k$ is a {\em $k$-cycle}. 
	We denote by $N^+_D(v)$ ($N^-_D(v)$, respectively) the set of  out-neighbours (in-neighbours, respectively) of $D$, i.e., $N^+_D(v)=\{u\in V(D): vu\in A(D)\}$ ($N^-_D(v)=\{u\in V(D): uv\in A(D)\}$, respectively). We use $N^+[v]$ ($N^-[v]$) to denote the the closed out-neighbourhood (in-neighbourhood) of $v$, i.e., $N^+[v]=N^+(v)\cup \{v\}$ ($N^-[v]=N^-(v)\cup \{v\}$). Note that $|N^+_D(v)|\le d_D^+(v)$ and $|N^-_D(v)|\le d_D^-(v)$ for a vertex $v$ in $D$. \LH{The oriented multigraph obtained by deleting a vertex (or arc) set $X$  to $D$ is denoted by $D-X$.}
The complementarity of $\fas(D)$ and $\acs(D)$ can be clearly viewed by ordering the vertices of $D$, $v_1,v_2,\dots ,v_n$, and classifying every arc $v_iv_j$ of $D$ as  a {\em forward arc} if $i<j$ or a {\em backward arc} if $i>j$. We can view the backward arcs as a feedback arc set and the forward arcs as the arcs of an acyclic subdigraph.
	
	\2
	
	This paper is organized as follows. 
	We prove Theorem \ref{thm:main} in the next section and Theorem \ref{thm:5-rg} in Section \ref{sec3}. In Section \ref{sec:examples}, we obtain lower bounds for
	$c''_{\leq 5}$, $c'_{\leq 6}$ and $c''_{\leq 6}$ and disprove Conjecture \ref{conj2}.

	\section{Proof of Theorem \ref{thm:main}}\label{sec:thm2}

	
	Let $D$ be an oriented multigraph with $\Delta(D)\leq 5$ \LH{and $x$ a vertex in $D$}.	Consider the following possible reductions, reducing $D$ to $D'$ \LH{and are referred to as {\em good} reductions.}
	
	\begin{description}
		\item[(g1):] $d_D(x)=1$. In this case let $D'=D-x$.
		
		\item[(g2a):] $d_D(x)=2$ and $d^+_D(x) \in \{0,2\}$. In this case let $D'=D-x$.
		
		\item[(g2b):] $d_D(x)=2$ and $d^+_D(x)=1$ and $x$ does not belong to a $3$-cycle in $D$.
		
		Let $N^+_D(x)=\{y\}$ and let $N^-_D(x)=\{z\}$. Now let $D'$ be obtained from $D$ by deleting $x$ and adding the arc $zy$.
		As $x$ does not belong to a $3$-cycle we note that $zy$ does not belong to a $2$-cycle in $D'$.
		
		\item[(g3a):] $d_D(x)=3$ and $d^+_D(x) \in \{0,3\}$. In this case let $D'=D-x$.
		
		\item[(g3b):] $d_D(x)=3$ and $d^+_D(x) \in \{1,2\}$ and $x$ is incident with 2 parallel arcs.
		
		Let $N^+_D(x)=\{y\}$ and let $N^-_D(x)=\{z\}$ (where either there are two parallel arcs from $z$ to $x$ or from $x$ to $y$).
		If $yz \in A(D)$, then let $D'$ be obtained from $D-x$ by removing \YZ{ one arc from $y$ to $z$}.
		If $yz \not\in A(D)$, then let $D'$ be obtained from $D-x$ by adding the arc $zy$.
		
		\item[(g4):] $d_D(x)=4$ and $d^+_D(x) \in \{0,1,3,4\}$.  In this case let $D'=D-x$.
		
		\item[(g5):] $d_D(x)=5$ and $d^+_D(x) \in \{0,1,4,5\}$.  In this case let $D'=D-x$.
	\end{description}
	
	 We also consider the following reductions, \LH{referred to as {\em nice} reductions.}
	
	\begin{description}
		\item[(n0):] $d_D(x)=0$. In this case let $D'=D-x$.
		
		\item[(n2):] $d_D(x)=2$ and $x$ belongs to a $3$-cycle in $D$. In this case let $C$ be the $3$-cycle containing $x$
		and let $D'$ be obtained from $D$ by deleting the arcs of $C$ (we may also delete $x$ as it is now an isolated vertex).
		
		\item[(n3):] $d_D(x)=3$ and $d^+_D(x) \in \{1,2\}$ and $x$ is not incident with parallel arcs.
		In this case let $D'=D-x$.
		
		\item[(n55):] $xy \in A(D)$, $d_D(x)=d_D(y)=5$ and either $d^+_D(x) = 2$ or $d^+_D(y)=3$ (or both).
		In this case let $D'=D-\LH{\{x,y\}}$.
		
		\item[(ntt):] \YZ{$\min\{d_D(x), d_D(y), d_D(z)\}\geq 4$}, $\{x,y,z\}$ forms a transitive-triangle in $D$ (i.e $xy,xz,yz \in A(D)$) and either $\max\{d^+_D(x),d^+_D(y),d^+_D(z)\} \leq 2$ or \\
		$\max\{d^-_D(x),d^-_D(y),d^-_D(z)\} \leq 2$. In this case let $D'=D-\LH{\{x,y,z\}}$.
		
		\item[(n545):] A vertex, $x$, with $d^+_D(x)=d^-_D(x)=2$ is adjacent to two non-adjacent degree-5 vertices \LH{$y$ and $z$} in $D$,
		such that $d^+_D(y)=d^+_D(z)$ and $d^+_D(y) \in \{2,3\}$.  In this case let $D'=D-\LH{\{x,y,z\}}$.
	\end{description}
	
	\begin{lemma} \label{lemA}
		If we perform a good reduction on $D$, resulting in a digraph $D'$, then the following \LH{hold}.
		\begin{description}
			\item[(A):] $|A(D')| \leq |A(D)| - (3k+1)$ and $\fas(D) \leq \fas(D') + k$ for some integer $k$.
			\item[(B):] $\Delta(D') \leq 5$ and $D'$ contains no $2$-cycles.
		\end{description}
	\end{lemma}
	
	\begin{proof}
		We will consider each reduction in turn.

		\2
		
		{\bf Reduction~(g1):} $d_D(x)=1$ for some $x \in V(D)$ and $D'=D-x$. Clearly $\fas(D)=\fas(D')$ and $|A(D')| = |A(D)| - 1$, so (A) holds. As $D'$ is a subdigraph of $D$ we note that (B) also holds.
		
		\2
		
		{\bf Reduction~(g2a):} $d_D(x)=2$ and $d^+_D(x) \in \{0,2\}$ for some $x \in V(D)$ and $D'=D-x$.  Clearly $\fas(D)=\fas(D')$ and
		$|A(D')| = |A(D)| - 2$, so (A) holds. As $D'$ is a subdigraph of $D$ we note that (B) also holds.
		
		\2

		{\bf Reduction~(g2b):} $d_D(x)=2$ and $d^+_D(x)=1$ for some $x \in V(D)$ and $x$ does not belong to a $3$-cycle on $D$.
		Now $N^+_D(x)=\{y\}$ and $N^-_D(x)=\{z\}$ and $D'$ is obtained from $D$ by deleting $x$ and adding the arc $zy$.
		
		Let $x_1,x_2,\ldots,x_{n-1}$ be an ordering of $D'$ with $\fas(D')$ backward arcs
		and let $z = x_a$ and let $y=x_b$. If $a>b$ then $zy$ is a backward arc and the ordering $x, x_1,x_2,\ldots,x_{n-1}$ implies that $\fas(D) \leq \fas(D')$ in this case. If $b>a$ then  $zy$ is a forward arc and the ordering
		$ x_1,x_2,\ldots,x_{a}, x , x_{a+1}, x_{a+2}, \ldots, x_{n-1}$ again implies that $\fas(D) \leq \fas(D')$ in this case.
		So in all cases $\fas(D) \leq \fas(D')$ and $|A(D')| = |A(D)| - 1$ (we remove 2 arcs and add 1 arc).
		
		Furthermore, as  $x$ does not belong to a $3$-cycle on $D$ we note that $D'$ contains no $2$-cycles and as no out-degree or in-degree
		is larger in $D'$ than in $D$ we note that $\Delta(D') \leq \Delta(D) \leq 5$. So both (A) and (B) hold.

		\2
		
		{\bf Reduction~(g3a):}  $d_D(x)=3$ and $d^+_D(x) \in \{0,3\}$ for some $x \in V(D)$ and $D'=D-x$. Clearly, $\fas(D)=\fas(D')$ and $|A(D')| = |A(D)| - 3$, so (A) holds. As $D'$ is a subdigraph of $D$ we note that (B) also holds.

		\2
		
		{\bf Reduction~(g3b):} $d_D(x)=3$ and $d^+_D(x) \in \{1,2\}$ for some $x \in V(D)$ and $x$ is incident with 2 parallel arcs.
		Let $N^+_D(x)=\{y\}$ and let $N^-_D(x)=\{z\}$ (where either there are two parallel arcs from $z$ to $x$ or from $x$ to $y$.
		If $yz \in A(D)$, then $D'$ is obtained from $D-x$ by removing \YZ{one arc from $y$ to $z$}.
		If $yz \not\in A(D)$, then $D'$ is obtained from $D-x$ by adding the arc $zy$.
		
		We first consider the case when $yz \in A(D)$ and note that $|A(D')| = |A(D)| - 4$.
		Assume we have an ordering of $V(D')$ with $\fas(D')$ backward arcs.
		
		If $y$ comes before $z$ in the ordering, then place $x$ either at the front (if there are parallel arcs from $x$ to $y$) or
		at the end (if there are parallel arcs from $z$ to $x$) of the ordering. This increases the number of backward arcs by at most one, so
		$\fas(D) \leq \fas(D')+1$ in this case.
		
		If $z$ comes before $y$ in the ordering, then place $x$ anywhere between $y$ and $z$ and note that $yz$ is the only backward arc
		added, so  $\fas(D) \leq \fas(D')+1$ in this case. As $\fas(D) \leq \fas(D')+1$ in all cases we note that (A) holds in this case.
		
		We now consider the case when $yz \not\in A(D)$ and note that $|A(D')| = |A(D)| - 2$ (as we delete 3 arcs and add the arc $zy$).
		If $zy$ is a forward arc in an optimal ordering of $D'$ then we add $x$ in between $z$ and $y$ and note that $\fas(D) \leq \fas(D')$.
		If $zy$ is a backward arc in an optimal ordering of $D'$ then again adding $x$ either at the front (if there are parallel arcs from $x$ to $y$) or
		at the end (if there are parallel arcs from $z$ to $x$) of the ordering shows that $\fas(D) \leq \fas(D')$ (as we remove the backward arc $zy$ but
		add a new backward arc). So again in all cases (A) holds.
		
		It is not difficult to see that (B) also holds.

		\2
		
		{\bf Reduction~(g4):} $d_D(x)=4$ and $d^+_D(x) \in \{0,1,3,4\}$ for some $x \in V(D)$ and $D'=D-x$.
		As we can add $x$ to the front or end of any ordering of $V(D')$ and increase the number of backward arcs by at most one
		we note that $\fas(D) \leq \fas(D')+1$. As $|A(D')| = |A(D)| - 4$, we note that
		(A) holds with $k=1$. As $D'$ is a subdigraph of $D$ we note that (B) also holds.

		\2
		
		{\bf Reduction~(g5):}  $d_D(x)=5$ and $d^+_D(x) \in \{0,1,4,5\}$ for some $x \in V(D)$ and $D'=D-x$.
		As we can add $x$ to the front or end of any ordering of $V(D')$ and increase the number of backward arcs by at most one
		we note that $\fas(D) \leq \fas(D')+1$. \LH{Since $|A(D')| = |A(D)| - 5$, 
		(A) holds with $k=1$}. As $D'$ is a subdigraph of $D$ we note that (B) also holds.
	\end{proof}
	
	\begin{lemma} \label{lemB}
		If we perform a \LH{nice} reduction
, resulting in a digraph $D'$, then the following \LH{hold}.
		\begin{description}
			\item[(A):] $|A(D')| \leq |A(D)| - 3k$ and $\fas(D) \leq \fas(D') + k$ for some integer $k$.
			\item[(B):] $\Delta(D') \leq 5$ and $D'$ contains no $2$-cycles.
			\item[(C):] $D'$ is a subdigraph of $D$.
		\end{description}
	\end{lemma}
	
	\begin{proof}
		We first note that all reductions just remove vertices (and the arcs incident with these vertices). So clearly (B) \LH{and (C) hold}
		for all reductions. We will consider each reduction in turn and prove part (A).
		
		\2
		
		{\bf Reduction~(n0):}  $d_D(x)=0$ for some $x \in V(D)$ and $D'=D-x$.   Clearly $\fas(D)=\fas(D')$ and $|A(D')| = |A(D)|$, so
		(A) holds.
		
		\2
		
		{\bf Reduction~(n2):} $d_D(x)=2$ for some $x \in V(D)$ and $x$ belongs to a $3$-cycle, $C$, in $D$ and let $D'$ be obtained from
		$D$ by deleting the arcs of $C$. Assume that $N^+_D(x)=\{y\}$ and $N^-_D(x)=\{z\}$. \YZ{Let $v_1,v_2,\dots, v_n$ be the ordering of $V(D)$ with $\fas(D')$ backward arcs in $D'$, and $v_i=y$ and $v_j=z$. If $i>j$, then we obtain an ordering of $V(D)$ by moving $x$ to anywhere between $y$ and $z$. If $i<j$,} then we move $x$ to the front of the ordering. Note that in each case we obtain an ordering of $V(D)$ with at most $\fas(D')+1$ backward arcs in $D$ and therefore $\fas(D)\leq \fas(D')+1$. And we have $|A(D')|\leq |A(D)|-3$, so (A) holds.

		\2
		
		{\bf Reduction~(n3):}  $d_D(x)=3$ for some $x \in V(D)$ and $d^+_D(x) \in \{1,2\}$ and $x$ is not incident with parallel arcs
		and $D'=D-x$. Assume that $v_1,v_2,\dots v_{n-1}$ is an ordering of $V(D')$ with $\fas(D')$ backward arcs in $D'$. By adding $x$ to the front (if $d^+_D(x)=2$) or to the end (if $d^+_D(x)=1$) of the ordering we have an ordering with at most $\fas(D')+1$ backward arcs in $D$ and therefore $\fas(D)\leq \fas(D')+1$. And as $|A(D)|\leq |A(D')|-3$, (A) holds.
		
		\2
		
		{\bf Reduction~(n55):}  $xy \in A(D)$, $d_D(x)=d_D(y)=5$ and either $d^+_D(x) = 2$ or $d^+_D(y)=3$ (or both) and $D'=D-\LH{\{x,y\}}$.
		Let $u_1,u_2,\ldots,u_{n-2}$ be an ordering of $V(D')$ with $\fas(D')$ backward arcs in $D'$. If $d^-_D(x)=d^-_D(y)=2$ then $x,y, u_1,u_2,\ldots,u_{n-2}$ is an ordering of $D$ with at most $\fas(D')+3$ backward arcs (we may
		add two backward arcs into $x$ and 1 into $y$). Analogously, if
		$d^+_D(x)=d^+_D(y)=2$ then $u_1,u_2,\ldots,u_{n-2},x,y$ is an ordering of $D$ with at most $\fas(D')+3$ backward arcs.
		Finally if $d^+_D(x) = 2$ and $d^+_D(y)=3$ (which implies that $d^-_D(y)=2$) then
		$y,u_1,u_2,\ldots,u_{n-2},x$ is an ordering of $D$ with at most $\fas(D')+3$ backward arcs (the arc $xy$ and one additional arc into $y$ and
		one additional arc out of $x$). So in all cases $\fas(D) \leq \fas(D')+3$.
		
		So if there are no parallel arcs from $x$ to $y$ then $|A(D')| \leq |A(D)| - 9$ and (A) holds.
		And if there are parallel arcs from $x$ to $y$ then the above orderings all show that $\fas(D) \leq \fas(D')+2$
		and we have $|A(D')| \leq |A(D)| - 8$ and (A) again holds.

		\2
		
		{\bf Reduction~(ntt):}  \YZ{$\min\{d_D(x), d_D(y), d_D(z)\}\geq4$,  $\{x,y,z\}$ forms a transitive-triangle in $D$ (i.e $xy,xz,yz \in A(D)$) and either $\max\{d^+_D(x),d^+_D(y),d^+_D(z)\} \leq 2$ or
			$\max\{d^-_D(x),d^-_D(y),d^-_D(z)\} \leq 2$} and $D'=D-\LH{\{x,y,z\}}$. Without loss of generality assume that $\max\{d^+_D(x),d^+_D(y),d^+_D(z)\} \leq 2$ and that
		$u_1,u_2,\ldots,u_{n-3}$ is an ordering of $V(D')$ with $\fas(D')$ backward arcs in $D'$. Consider the ordering $u_1,u_2,\ldots,u_{n-3}, x,y,z$ of $D$. We add at most 2 backward arcs out of $z$ and at most $1$ backward arc
		out of $y$ and no backward arc out of $x$. So, if there are no parallel arcs between vertices in  $\{x,y,z\}$ then we note that
		$\fas(D)\leq\fas(D') + 3$ and $|A(D')| \leq |A(D)| - 9$, so (A) holds.
		
		If there are parallel arcs connecting vertices of $\{x,y,z\}$, \LH{then the parallel arcs must go from $y$ to $z$ as $d^+_D(x) \leq 2$}.
		And as $d^+_D(y) \leq 2$ there can be at most 2 parallel arcs from $y$ to $z$. \LH{As stated above}, we now have
		$\fas(D)\leq \fas(D') + 2$ (as no backward arcs will leave $y$)  and $|A(D')| \leq |A(D)| - 8$, so (A) again holds.

		\2
		
		{\bf Reduction~(n545):}  A vertex, $x$, with $d^+_D(x)=d^-_D(x)=2$ is adjacent to two non-adjacent  degree-5 vertices \LH{$y$ and $z$} in $D$,
		such that $d^+_D(y)=d^+_D(z)$ and $d^+_D(y) \in \{2,3\}$ and $D'=D-\LH{\{x,y,z\}}$.
		Let $\sigma_{D'}$ be an ordering of $V(D')$ with $\fas(D')$ backward arcs in $D'$.
		
		We only consider the case when  $d^+_D(y)=d^+_D(z)=2$ as the case when $d^-_D(y)=d^-_D(z)=2$ can be proved analogously. We place an acyclic ordering (with no backward arcs) of $x,y,z$ at the end of the ordering $\sigma_{D'}$
		and note that there are at most 4 arcs from $\{x,y,z\}$ to $V(D')$ in $D$. So $\fas(D) \leq \fas(D') + 4$. In addition, \LH{if} there are no parallel arcs connecting vertices within $\{x,y,z\}$ then $|A(D')| = |A(D)| - 12$ and so (A) holds.
		And if there are parallel arcs within $\{x,y,z\}$ then $|A(D')| \leq |A(D)| - 10$ and $\fas(D) \leq \fas(D') + 3$ so again (A) holds.
	\end{proof}
	
	Now we are ready to prove Theorem \ref{thm:main}. For convenience of the reader, let us formulate it first.
	
	\vspace{1mm}
	
		\noindent {\bf Theorem \ref{thm:main}.}
		{\em		If $D$ is an oriented multigraph with $\Delta\leq 5$, then $\fas(D) \leq m /3$.}
	
	\begin{proof}
		We will prove the theorem by induction on $|A(D)|$. If $|A(D)| \leq 2$, then $\fas(D)=0$ and the theorem holds.
		Now assume that $|A(D)| \geq 3$ and the theorem holds for all digraphs of smaller size.
		If we can perform any of \LH{the good and nice} reductions in order to obtain $D'$, then the following holds for some integer $k$, by Lemma~\ref{lemA} and
		Lemma~\ref{lemB}.
		\[
		\fas(D) \leq \fas(D') + k \leq \frac{|A(D')|}{3} + k \leq \frac{|A(D)|-3k}{3} + k = \frac{|A(D)|}{3}.
		\]
		
		So, we may assume that none of \LH{the good and nice reductions} can be performed on $D$. This implies that $4 \leq d(x) \leq 5$ and
		$d^+(x), d^-(x) \in \{2,3\}$ for all $x \in V(D)$.
		Let $X_5^+$ contain all vertices with degree five and out-degree three and
		let $X_5^-$ contain all vertices with degree five and out-degree two and
		let $X_4$ contain all vertices with degree four \YZ{and out-degree two} in $D$.
		Note that $(X_5^+,X_4,X_5^-)$ is a partition of $V(D)$.
		Furthermore $X_5^-$ and $X_5^+$ are both independent sets and there are no arcs from $X_5^-$ to $X_5^+$ as otherwise we could use reduction (n55).
		We will now prove the following claim.
		
		\2
		
		{\bf Claim A:} {\em If $D'$ is a subdigraph of $D$ with a vertex of degree three then we can perform a sequence of \LH{(n2) or
			(n3) reductions} on $D'$, resulting in a digraph $D^*$, such that we can perform a good reduction on $D^*$.}
		
		\2
		
		{\bf Proof of Claim A:} We will prove the claim by induction on the size of $D'$. If $|A(D')| \leq 3$, then
		$|A(D')| = 3$ as $D'$ contains a vertex of degree 3. \YZ{Let $D^*=D'$} and we can perform reduction~(g1), (g2a) or (g3a) on $D^*$ respectively if $D^*$ has exactly one, two or three parallel arcs. So we may assume that \AY{$|A(D')| \geq 4$} and the claim holds for all smaller subdigraphs of $D$.
		
		Let $q$ be a vertex of degree three in $D'$. We may assume that $d_{D'}^+(q),d_{D'}^-(q) \in \{1,2\}$ and
		$q$ is not incident with any parallel arcs, as otherwise \YZ{let $D^*=D'$} and we can perform reduction~(g3a) or (g3b) on $D^*$.
		Let $a,b,c$ be the vertices in $D'$ adjacent to $q$.
		
		For the sake of contradiction assume that $\{a,b,c\} \subseteq X_5^+ \cup X_5^-$. As
		$X_5^-$ and $X_5^+$ are both independent sets and there is no arc from $X_5^-$ to $X_5^+$ we note that
		$q \not\in X_5^+ \cup X_5^-$ (as $q$ has arcs to and from $\{a,b,c\}$).
		This contradicts the fact that we cannot perform reduction~(n545) on $D$.
		So, $\{a,b,c\} \not\subseteq X_5^+ \cup X_5^-$.
		
		Therefore some vertex $v \in \{a,b,c\}$ has degree four in $D$ and therefore degree at most three in $D'-q$.
		If $v$ has degree three in $D'-q$, then by induction there is a
		sequence of \LH{(n2) or
			(n3) reductions} on $D'-q$ such that we can perform a good reduction on the resulting digraph \YZ{$D^*$},
		which completes the proof in this case as $D'-q$ is a (n3)-reduction (since $d_{D'}^+(q),d_{D'}^-(q) \in \{1,2\}$ and $q$ is not incident with any parallel arcs). So $v$ has degree at most $2$ in $D'-q$. We may assume that $v$ has degree exactly two and belongs to a $3$-cycle $C$ in $D'-q$, as otherwise \YZ{let $D^*=D'-q$ and we can perform reduction (g1), (g2a) or (g2b) on $D^*$. Thus, $d_{D'}(v)=3$ as $v$ is adjacent to $q$. Note that we may also assume that $d_{D'}^+(v),d_{D'}^-(v) \in \{1,2\}$ and $v$ is not incident with any parallel arcs for the same reason as $q$.}
		
		We first consider the case when $V(C) \not= \{a,b,c\}$. Let $u \in V(C) \setminus \{a,b,c\}$ be arbitrary, $D_1 = D'-v$ (which is a \AY{(n3)-reduction} on $D'$) and $D_2=D'-q -A(C)$ (which is a (n3)-reduction followed by a (n2)-reduction on $D'$). Observe that $d^+_{D_1}(u)=d^+_{D'}(u)-1$ and $d^+_{D_2}(u)=d^+_{D'}(u)-2$. So, in $D_1$ or $D_2$ the degree of $u$ is odd and at most three. If it is of degree one, then we can perform reduction (g1). Otherwise, $u$ has degree 3 and therefore we are done by induction. This completes the case when $V(C) \not= \{a,b,c\}$.

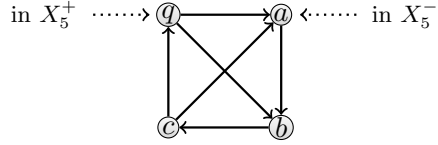
\begin{figure}[t]
\begin{center}
\tikzstyle{vertexY}=[circle,draw, top color=gray!5, bottom color=gray!30, minimum size=8pt, scale=0.99, inner sep=0.1pt]
\begin{tikzpicture}[scale=0.5]
\node (q) at (1,4) [vertexY] {$q$};
\node (a) at (4,4) [vertexY] {$a$};
\node (b) at (4,1) [vertexY] {$b$};
\node (c) at (1,1) [vertexY] {$c$};


\draw [->, line width=0.03cm] (q) -- (a);
\draw [->, line width=0.03cm] (q) -- (b);
\draw [->, line width=0.03cm] (c) -- (q);
\draw [->, line width=0.03cm] (a) -- (b);
\draw [->, line width=0.03cm] (b) -- (c);
\draw [->, line width=0.03cm] (c) -- (a);

\draw [->, dotted, line width=0.03cm] (-1,4) -- (0.5,4);
\draw (-2.3,4) node {{\footnotesize in $X_5^+$}};

\draw [->, dotted, line width=0.03cm] (6,4) -- (4.5,4);
\draw (7.3,4) node {{\footnotesize in $X_5^-$}};
\end{tikzpicture}

\caption{
Illustration of one of the cases in the proof of Theorem~\ref{thm:main}.} \label{fig:ThmMain}
\end{center} \end{figure}

		So, we now assume that $V(C) = \{a,b,c\}$.
		We may without loss of generality assume that $q a,q b, c q \in A(D')$ and $C=abca$.
		As $\{q,a,b,c\}$ \LH{forms} a clique in $D'$ (and therefore also in $D$) we note that it contains at most
		one vertex from $X_5^+$ and at most one vertex from $X_5^-$.
		As both $\{q,a,b\}$ and $\{q,a,c\}$ form transitive-triangles in $D$ we note that both of these
		sets contain a vertex from $X_5^+$ and a vertex from $X_5^-$ (as otherwise we could have performed
		reduction~(ntt) on $D$). As there are no arcs from $X_5^-$ to $X_5^+$ in $D$ we note that $q \in X_5^+$
		and $a \in X_5^-$ \AY{(see Figure~\ref{fig:ThmMain})}. This implies that $b$ has degree four in $D$ and therefore degree at most three in $D'-q$.
		If it has degree three in $D'-q$, then \AY{we are done by using induction on $D'-q$}. \YZ{If   $b$ has  degree at most two in $D'-q$,
			then this implies that it has degree three in $D'$ as $\{q,a,b,c\}$ forms a clique in $D'$}.
		
		Now let  $D^*=D'-b$ (which is a (n3)-reduction in $D'$). Note that $q$ has degree two and  does not belong to a $3$-cycle in $D^*$ (as $\{c,q,a\}$ form a transitive-triangle in $D'-b$), so we can perform reduction~(g2b) in $D^*$, which completes the
		proof of the claim.

		\2

		We now return to the proof of the theorem. Note that the following holds.

		\[
		0 = |A(D)|-|A(D)| =  \sum_{u \in V(D)} (d^+(u) - d^-(u)) = |X_5^+| - |X_5^-|.
		\]
		
		The above implies that $|X_5^+| = |X_5^-|$. If $|X_5^-|=0$ then $D$ is degree-$4$ and by Theorem \ref{thm:md34}, the theorem holds. Thus, we may assume that $|X_5^-|>0$.
		Let $x \in X_5^-$ be arbitrary and let $D' = D - x$.
		As there are no arcs from $X_5^-$ to $X_5^+$ and $X_5^-$ is independent there \LH{exists an arc $xy \in A(D)$ such that} the degree of $y$
		is $4$ in $D$. If there are two parallel arcs from $x$ to $y$, then let $D^* = D - x$ and note that we can perform the good reduction (g2a) on $D^*$ by deleting $y$. And, if there are \LH{no} two parallel arcs from $x$ to $y$ in $D$ then by Claim~A there exists a sequence of \LH{(n2) or
			(n3) reductions} on $D'$, resulting in a digraph $D^*$, such that we can perform a good reduction on $D^*$.
		Note that every \LH{(n2)-reduction or (n3)-reduction} removes exactly $3$ arcs and decreases the feedback arc set by at most one.
		Furthermore, note that  $\fas(D) \leq \fas(D') + 2$ and $|A(D')| = |A(D)| - 5$. So if we have performed $r$ such (n2)-reductions and (n3)-reductions, then the following holds.

		\[
		\begin{array}{rcl}
			|A(D^*)| & = & |A(D)| - 5 -3r. \\
			\fas(D) & \leq & \fas(D^*) + 2 + r. \\
		\end{array}
		\]
		
		We can now perform a good reduction on $D^*$ and obtain a new digraph $D^+$. The following now \LH{hold}
		for some integer $k$.
		
		\[
		\begin{array}{rcl}
			|A(D^+)| & \leq &  |A(D^*)| - (3k+1) \\
			& = & (|A(D)| - 5 -3r) - (3k+1) \\
			& = & |A(D)| -3r-3k-6. \\
			&   & \\
			\fas(D) & \leq &  \fas(D^*) + 2 + r  \\
			& \leq & (\fas(D^+) + k) + 2 + r.  \\
		\end{array}
		\]
		
		Therefore the following holds, by induction (on $D^+$),
		
		\[
		\fas(D) \leq  \fas(D^+) + k + 2 + r   \leq \frac{|A(D^+)|}{3} + k + 2 + r  \leq \frac{|A(D)|}{3}.
		\]
		
		This completes the proof of the theorem.
	\end{proof}
	
	\section{Proof of Theorem \ref{thm:5-rg}}\label{sec3}
	Let $D$ be a degree-$5$ oriented multigraph and let $u\in V(D)$ be arbitrary. Define $Q(u)$ as follows.
	\[
	Q(u)=\left\{
	\begin{aligned}
		&\{u\}, &if~d^+(u)\in \{0, 1, 4, 5\}; \\
		&N^+[u], &if~d^+(u)=2; \\
		&N^-[u], &if~d^-(u)=2. \\
	\end{aligned}
	\right.
	\]
	
	\begin{lemma}\label{lem:5-regular}
		Let D be a degree-$5$ oriented multigraph and let $S\subseteq V(D)$ be chosen such that the following holds. For all distinct $u, v\in S$ we have $Q(u)\cap Q(v) = \emptyset$ and there are no arcs between $Q(u)$ and $Q(v)$. Then, $\fas(D)\leq \frac{|A(D)|-|S|}{3}$.
	\end{lemma}
	\begin{proof}
		Let $S_0$ be those $u\in S$ with $Q(u)=\{u\}$. Let $S^+_1$ ($S^+_2$, resp.) be the set of those $u\in S$ with $Q(u)=N^+[u]$ and one out-neighbour (two out-neighbours, resp.) in $D$. Let $S^-_1$ ($S^-_2$, resp.) be the set of those $u\in S$ with $Q(u)=N^-[u]$ and one in-neighbour (two in-neighbours, resp.) in $D$. Note that $S=S_0\cup S^+_1\cup S^-_1\cup S^+_2\cup S^-_2$. Let $D'=D-\cup_{u \in S}Q(u)$. Then, by Theorem \ref{thm:main} there is an ordering $\sigma_{D'}$ of the vertices in $D'$ with at most $|A(D')|/3$ arcs. As for any pair of distinct vertices $u$ and $v$, $Q(u)\cap Q(v)=\emptyset$ and there are no arcs between $Q(u)$ and $Q(v)$, we have
		\[\frac{|A(D')|}{3}=\frac{|A(D)|-5|S_0|-8|S^+_1\cup S^-_1|-13|S^+_2\cup S^-_2|}{3}.\]
		Now, we extend $\sigma_{D'}$ to an ordering of $V(D)$ by doing the following for every vertex $u\in S$.

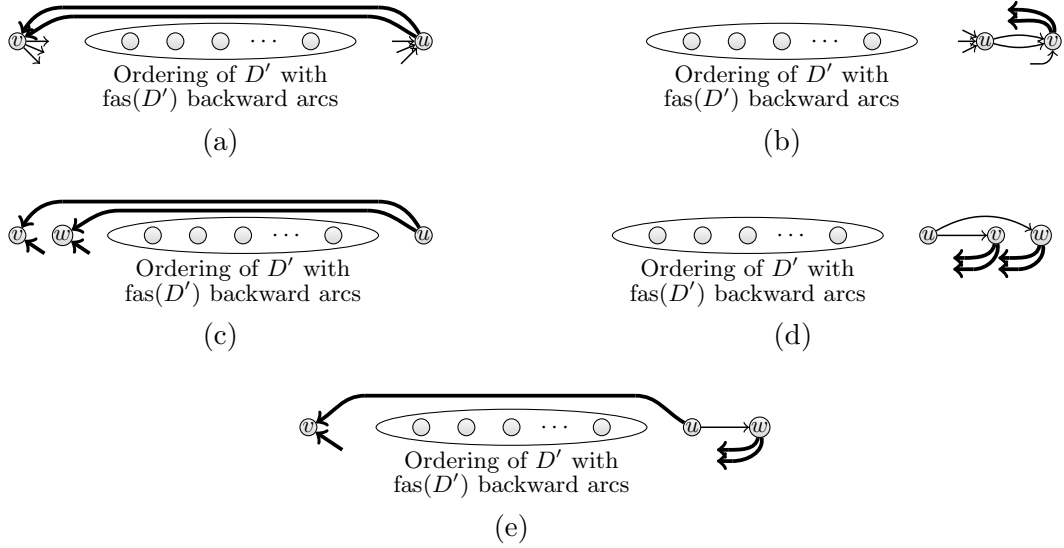
\begin{figure}[t]
\begin{center}
\tikzstyle{vertexY}=[circle,draw, top color=gray!5, bottom color=gray!30, minimum size=8pt, scale=0.8, inner sep=0.1pt]
\begin{tikzpicture}[scale=0.3]
\node (v) at (3,1) [vertexY] {$v$};
\node (u) at (21,1) [vertexY] {$u$};

\node (x1) at (8,1) [vertexY] {};
\node (x2) at (10,1) [vertexY] {};
\node (x3) at (12,1) [vertexY] {};
\draw (14,1) node {{\footnotesize $\cdots$}};
\node (xn) at (16,1) [vertexY] {};
\draw (12,1) ellipse (6 and 0.8);

\draw [line width=0.05cm] (u) to [out=150, in=0] (18.5,2.1);
\draw [line width=0.05cm] (18.5,2.1) -- (5.5,2.1);
\draw [->, line width=0.05cm] (5.5,2.1) to [out=180, in=30] (v);

\draw [line width=0.05cm] (u) to [out=120, in=0] (19,2.5);
\draw [line width=0.05cm] (19,2.5) -- (5,2.5);
\draw [->, line width=0.05cm] (5,2.5) to [out=180, in=60] (v);

\draw [->, line width=0.02cm] (20,0) -- (u);
\draw [->, line width=0.02cm] (19.8,0.5) -- (u);
\draw [->, line width=0.02cm] (19.6,1) -- (u);

\draw [->, line width=0.02cm] (v) -- (4,0);
\draw [->, line width=0.02cm] (v) -- (4.2,0.5);
\draw [->, line width=0.02cm] (v) -- (4.4,1);

\draw (12,-0.5) node {{\footnotesize Ordering of $D'$ with}};
\draw (12,-1.6) node {{\footnotesize $\fas(D')$ backward arcs}};

\draw (12,-3.5) node {(a)};
\end{tikzpicture} \hfill
\begin{tikzpicture}[scale=0.3]
\node (u) at (21,1) [vertexY] {$u$};
\node (v) at (24,1) [vertexY] {$v$};

\node (x1) at (8,1) [vertexY] {};
\node (x2) at (10,1) [vertexY] {};
\node (x3) at (12,1) [vertexY] {};
\draw (14,1) node {{\footnotesize $\cdots$}};
\node (xn) at (16,1) [vertexY] {};
\draw (12,1) ellipse (6 and 0.8);

\draw [->, line width=0.02cm] (u) to [out=15, in=165] (v);
\draw [->, line width=0.02cm] (u) to [out=-15, in=195] (v);

\draw [->, line width=0.02cm] (20,0.5) -- (u);
\draw [->, line width=0.02cm] (19.8,1) -- (u);
\draw [->, line width=0.02cm] (20,1.5) -- (u);

\draw [line width=0.05cm] (v) to [out=80, in=0]  (22.5,2.4);
\draw [->, line width=0.05cm] (22.5,2.4) -- (22,2.4);

\draw [line width=0.05cm] (v) to [out=100, in=0]  (22.2,1.9);
\draw [->, line width=0.05cm] (22.2,1.9) -- (22,1.9);

\draw [->, line width=0.02cm] (23.5,0) to [out=0, in=270] (v);
\draw [line width=0.02cm] (23,0) -- (23.5,0);

\draw (12,-0.5) node {{\footnotesize Ordering of $D'$ with}};
\draw (12,-1.6) node {{\footnotesize $\fas(D')$ backward arcs}};

\draw (12,-3.5) node {(b)};
\end{tikzpicture}

\vspace{0.4cm}

\begin{tikzpicture}[scale=0.3]
\node (v) at (3,1) [vertexY] {$v$};
\node (w) at (5,1) [vertexY] {$w$};

\node (u) at (21,1) [vertexY] {$u$};

\node (x1) at (9,1) [vertexY] {};
\node (x2) at (11,1) [vertexY] {};
\node (x3) at (13,1) [vertexY] {};
\draw (15,1) node {{\footnotesize $\cdots$}};
\node (xn) at (17,1) [vertexY] {};
\draw (13,1) ellipse (6 and 0.8);

\draw [line width=0.05cm] (u) to [out=150, in=0] (18.5,2.1);
\draw [line width=0.05cm] (18.5,2.1) -- (7.5,2.1);
\draw [->, line width=0.05cm] (7.5,2.1) to [out=180, in=30] (w);

\draw [line width=0.05cm] (u) to [out=120, in=0] (19,2.5);
\draw [line width=0.05cm] (19,2.5) -- (5,2.5);
\draw [->, line width=0.05cm] (5,2.5) to [out=180, in=60] (v);

\draw [->, line width=0.05cm] (4.2,0.2) -- (v);
\draw [->, line width=0.05cm] (6.2,0.2) -- (w);

\draw (13,-0.5) node {{\footnotesize Ordering of $D'$ with}};
\draw (13,-1.6) node {{\footnotesize $\fas(D')$ backward arcs}};

\draw (12,-3.5) node {(c)};
\end{tikzpicture} \hfill
\begin{tikzpicture}[scale=0.3]
\node (v) at (24,1) [vertexY] {$v$};
\node (w) at (26,1) [vertexY] {$w$};

\node (u) at (21,1) [vertexY] {$u$};

\node (x1) at (9,1) [vertexY] {};
\node (x2) at (11,1) [vertexY] {};
\node (x3) at (13,1) [vertexY] {};
\draw (15,1) node {{\footnotesize $\cdots$}};
\node (xn) at (17,1) [vertexY] {};
\draw (13,1) ellipse (6 and 0.8);

\draw [->, line width=0.02cm] (u) -- (v);
\draw [->, line width=0.02cm] (u) to [out=30, in=150] (w);

\draw [line width=0.05cm] (v) to [out=260, in=0] (22.6,0);
\draw [->, line width=0.05cm] (22.6,0) -- (22.1,0);

\draw [line width=0.05cm] (v) to [out=280, in=0] (23,-0.5);
\draw [->, line width=0.05cm] (23,-0.5) -- (22.1,-0.5);

\draw [line width=0.05cm] (w) to [out=260, in=0] (24.6,0);
\draw [->, line width=0.05cm] (24.6,0) -- (24.1,0);

\draw [line width=0.05cm] (w) to [out=280, in=0] (25,-0.5);
\draw [->, line width=0.05cm] (25,-0.5) -- (24.1,-0.5);

\draw (13,-0.5) node {{\footnotesize Ordering of $D'$ with}};
\draw (13,-1.6) node {{\footnotesize $\fas(D')$ backward arcs}};

\draw (15,-3.5) node {(d)};
\end{tikzpicture} \hfill

\vspace{0.4cm}

\begin{tikzpicture}[scale=0.3]
\node (v) at (4,1) [vertexY] {$v$};

\node (u) at (21,1) [vertexY] {$u$};
\node (w) at (24,1) [vertexY] {$w$};

\node (x1) at (9,1) [vertexY] {};
\node (x2) at (11,1) [vertexY] {};
\node (x3) at (13,1) [vertexY] {};
\draw (15,1) node {{\footnotesize $\cdots$}};
\node (xn) at (17,1) [vertexY] {};
\draw (13,1) ellipse (6 and 0.8);

\draw [->, line width=0.02cm] (u) -- (w);
\draw [line width=0.05cm] (u) to [out=150, in=0] (18.5,2.4);
\draw [line width=0.05cm] (18.5,2.4) -- (6.5,2.4);
\draw [->, line width=0.05cm] (6.5,2.4) to [out=180, in=30] (v);

\draw [->, line width=0.05cm] (5.5,0) -- (v);

\draw [line width=0.05cm] (w) to [out=260, in=0] (22.6,0);
\draw [->, line width=0.05cm] (22.6,0) -- (22.1,0);

\draw [line width=0.05cm] (w) to [out=280, in=0] (22.8,-0.5);
\draw [->, line width=0.05cm] (22.8,-0.5) -- (22.1,-0.5);

\draw (13,-0.5) node {{\footnotesize Ordering of $D'$ with}};
\draw (13,-1.6) node {{\footnotesize $\fas(D')$ backward arcs}};

\draw (13,-3.5) node {(e)};
\end{tikzpicture}

\caption{
Illustration of the different cases in Lemma~\ref{lem:5-regular}.  The thick arcs denote backward arcs.} \label{fig:5-regular}
\end{center} \end{figure}

		{\bf Case 1.} $d^+(u)\geq 4$ or $d^-(u)\geq 4$ (i.e. $u\in S_0$). In this case, we add $u$ to the front of the ordering if $d^+(u)\geq 4$ and to the end of the ordering if $d^-(u)\geq 4$. Clearly, in both cases, we add at most one new backward arc.
		
		{\bf Case 2.} $u\in S^+_1\cup S^-_1$. We only consider the case when $u\in S^+_1$ as the other is similar.
Assume that there are two parallel arcs from $u$ to $v$.
If $d^+(v)=3$, then we add $u$ to the end of the ordering and $v$ to the front of the ordering \AY{(see Figure~\ref{fig:5-regular}(a))}.
If $d^+(v)\leq 2$, then we add \AY{$u$ to the end of the ordering and $v$ after $u$ (see Figure~\ref{fig:5-regular}(b))}. In each case, we add at most two new backward arcs.
		
		{\bf Case 3.} $u\in S^+_2\cup S^-_2$. We only consider the case when $u\in S^+_2$ as the proof for the other one is similar. Assume that $N^+(u)=\{v,w\}$. We consider the following subcases.
		
		{\bf Subcase 3.1.} $d^+(v)\geq 3$ and $d^+(w)\geq 3$. In this case, we add $v$ and $w$ to the front of the ordering and $u$ to the end of the
		the ordering \AY{(see Figure~\ref{fig:5-regular}(c))}. Observe that we add at most four backward arcs.
		
		{\bf Subcase 3.2.} $d^+(v)\leq 2$ and $d^+(w)\leq 2$. Then, we add \AY{$u$} to the end of the ordering and put $v$ and $w$ after $u$ \AY{(see Figure~\ref{fig:5-regular}(d))}. Note that we add at most four backward arcs as $d^+(v)+d^+(w)\leq 4$.
		
		{\bf Subcase 3.3.} One of $v$ and $w$ has out-degree at most 2, and the other has out-degree at least 3. Assume without loss of generality that $d^+(v)\geq 3$ and $d^+(w)\leq 2$. Then we add $u$ to the end of the ordering and $w$ after $u$, and $v$ to the front of the ordering \AY{(see Figure~\ref{fig:5-regular}(e))}. One can observe that we have added at most four backward arcs (at most two arcs from $w$ and at most two arcs to $v$).
		
		Note that as for every pair of distinct vertices $u$ and $v$, $Q(u)\cap Q(v)=\emptyset$ and there are no arcs between $Q(u)$ and $Q(v)$ and therefore in particular there is no backward arc between $Q(u)$ and $Q(v)$ in the new ordering. Thus, after considering all vertices in $u$, we have an ordering of $V(D)$ in $D$, whose number of backward arcs is at most the following.
		\begin{eqnarray*}
			&&\frac{|A(D)|-5|S_0|-8|S^+_1\cup S^-_1|-13|S^+_2\cup S^-_2|}{3}+ |S_0|+2|S^+_1\cup S^-_1|+4|S^+_2\cup S^-_2|\\
			&=&\frac{|A(D)|-2|S_0|-2|S^+_1\cup S^-_1|-|S^+_2\cup S^-_2|}{3}\\
			&\leq & \frac{|A(D)|-|S|}{3},
		\end{eqnarray*}
		which completes the proof.
	\end{proof}
	
	\2
	
	For any digraph $D$ and any vertex $S\subseteq V(D)$, we use $N_D(S)$ to denote the set of in- and out-neighbours of $S$ in $D$, i.e., $N_D(S)=\{u\in V(D)\setminus S: uv\in A(D)~or~vu\in V(D)$ for some vertices $v\in S\}$. Let $N_D^2(S)=N_D(N_D(S))\setminus S$. Now we are ready to prove Theorem \ref{thm:5-rg}.

	\2
	
	\noindent{\bf Theorem \ref{thm:5-rg}.} {\em If $D$ is a degree-$5$ oriented multigraph, then $\fas (D)\leq \YZ{24n/29}$.}
	\begin{proof}	\YZ{Let $n=|V(D)|.$} As $D$ is degree-$5$, $|A(D)|= 5n/2$. \YZ{Note that $(5/2-1/58)/3=24/29.$}
		Thus, by Lemma \ref{lem:5-regular}, we only need to \LH{show that} there is a set $S$ \YZ{satisfying} the conditions in Lemma \ref{lem:5-regular} and $|S|\geq n/58$.
		
		Consider an auxiliary graph $H$ with vertex set $V(H)=V(D)$ where $uv\notin E(H)$ if and only if $Q(u)\cap Q(v)=\emptyset$ and there is no arcs between $Q(u)$ and $Q(v)$ in $D$. Note that any independent set in $H$ is a set \LH{that satisfies} the conditions in Lemma \ref{lem:5-regular}. And the independent number of $H$ is at least $|V(H)|/\chi(H)$. Thus, by Brooks' Theorem, we only need to show $\Delta(H)\leq 57$.
		
		{\bf Claim A:} Let $u$ be an arbitary vertex in $V(D)$. If $uu'\in E(H)$ then $u'\in N_D(u)\cup N_D^2(u)\cup N_D^2(Q(u)\setminus \{u\})$. In particular, $d_H(u)\leq 25+|N_D^2(Q(u)\setminus \{u\})\setminus N_D(u)|$.
		
		{\bf Proof of Claim A:} By the definition, we have $uu'\in E(H)$ if and only if $Q(u)\cap Q(u')\neq \emptyset$ or $N_D(u)\cap Q(u')\neq \emptyset$ or $N_D(Q(u)\setminus \{u\})\cap Q(u')\neq \emptyset$. \YZ{Note that $(Q(u)\cap Q(u'))\cup (N_D(u)\cap Q(u'))\neq \emptyset$ if and only if $N_D[u]\cap Q(u')\neq \emptyset$ which implies that $u'\in  N_D(u)\cup N_D^2(u)$}. And $N_D(Q(u)\setminus \{u\})\cap Q(u')\neq \emptyset$ implies that $u'\in N_D^2(u)\cup N_D^2(Q(u)\setminus \{u\})$. Clearly, $|N_D(u)|+|N_D^2(u)|\leq 5+5\times 4=25$. This completes the proof of Claim A.
		
		Recall that $|Q(u)|\in \{1,2,3\}$ for every $u\in V(D)$. If $Q(u)=\{u\}$, then $|N_D^2(Q(u)\setminus \{u\})\setminus N_D(u)|=0$ and therefore $d_H(u)\leq 25$. If $|Q(u)|=2$, then let $Q(u)\setminus \{u\}=\{v\}$. Since there are two parallel arcs in $D$ between $u$ and $v$, $|N_D^2(Q(u)\setminus \{u\})\setminus N_D(u)|=|N_D^2(v)\setminus N_D(u)|\leq 3\times 4=12$. Thus, by Claim A, $d_H(u)\leq 37$.  If $|Q(u)|=3$, then let $N_D(u)=\{v,w\}$. Thus, $|N_D^2(Q(u)\setminus \{u\})\setminus N_D(u)|\leq |N_D^2(v)\setminus N_D(u)|+|N_D^2(w)\setminus N_D(u)|\leq 2\times 4\times 4=32$ and therefore $d_H(u)\leq 25+32\leq 57$, which completes the proof. \end{proof}
	
	\section{Lower bounds for $c''_{\leq 5}$, $c'_{\leq 6}$ and $c''_{\leq 6}$}\label{sec:examples}
	
	We will obtain lower bounds for $c''_{\leq 5}$, $c'_{\leq 6}$ and $c''_{\leq 6}$ in the following two propositions. Note that all constructions are oriented graphs.
	\begin{proposition}\label{prop1}
\YZ{We have $c''_5=c''_{\leq 5}\geq 5/7$.}
	\end{proposition}
	\begin{proof}
		We first construct an oriented graph, $D_7$, with $|V(D_7)|=7$, $\Delta(D_7)= 5$ and $\fas(D_7) \geq 5$.
		Let $V(D_7)=\{u_1,u_2,\ldots, u_7\}$ and let $A(D)=A_1 \cup A_2 \cup A_3$, where $A_1$, $A_2$ and $A_3$ are defined as follows (where all indices are taken modulo 7).
		\[
		\begin{array}{rcl}
			A_1 & = & \{u_i u_{i+1} \; | \; i=1,2,3,4,5,6,7 \}; \\
			A_2 & = & \{u_{i+2} u_i \; | \; i=1,2,3,4,5,6,7 \}; \\
			A_3 & = & \{u_1 u_5, u_2 u_6 \}. \\
		\end{array}
		\]
		Note that there \LH{exist} the following nine $3$-cycles in $D_7$.
		\[
		T=\{ u_i u_{i+1} u_{i+2} u_i \; | \; i=1,2,3,4,5,6,7 \}  \cup \{u_1 u_5 u_3 u_1, u_2 u_6 u_4 u_2 \}
		\]
				Furthermore any arc in $D_7$ belongs to at most two of the nine $3$-cycles in $T$.
		Therefore $\fas(D_7) \geq \frac{9}{2}$ which implies that $\fas(D_7) \geq 5$ as $\fas(D_7)$ is an integer.
		This completes the construction of $D_7$.
\YZ{		Now we have $\fas(D_{7}) \geq 5 = \frac{5 n}{7}$. Hence, $c''_{\leq 5}\geq 5/7$. By Proposition \ref{relations}(i), $c''_5=c''_{\leq 5}\geq 5/7$.	}
	\end{proof}
	
	\2
	
	\begin{corollary}\label{rem}
		There is a counterexample to Conjecture \ref{conj2}.
	\end{corollary}
	\begin{proof}
\AY{Note that we can take two copies of $D_7$, defined in the proof of Proposition~\ref{prop1}, and add arcs between them in order to obtain a strong degree-5 oriented graph $D_{14}$.}
Recall that Conjecture \ref{conj2} states that for every strongly connected oriented graph $D$ with $\Delta(D)\le 5$, $\fas(D)\le 2n/3.$ Note \AY{that $D_{14}$} provides a counterexample to this conjecture as \AY{$\fas(D_{14}) \geq 2 \cdot \fas(D_7) = 10$ and} \YZ{$2/3<5/7$}.
	\end{proof}
	
	
	\begin{proposition}
		There exists a degree-$6$ oriented graph $D$ with $\fas(D)\geq 25m/72$. In particular, $c'_{\leq 6} \geq 25/72$ and \AY{$c''_{\leq 6} = c''_6 \geq 75/72$.}
	\end{proposition}
	\begin{proof}
		Let us take a copy of $D_7$ defined in the proof of Proposition \ref{prop1} where $V(D_7)=\{u_1,u_2,\ldots, u_7\}$ and $A(D_7)=A_1 \cup A_2 \cup A_3$. We then add a vertex $u_8$ and the \LH{arc set} $A_4 = \{u_5 u_8, u_8 u_1, u_6 u_8, u_8 u_2 \}$. The resulting oriented graph is denoted by $D_8$.
		
		Recall that $T$ is a set of nine $3$-cycles in $D_7$.  Define $T^*$ as follows.		
		\[
		T^* = T \cup \{u_5 u_8 u_1 u_5, u_6 u_8 u_2 u_6, u_7 u_5 u_8 u_2 u_7, u_1 u_6 u_8 u_1 \}.
		\]
		Note that $T^*$ contains 13 cycles (twelve $3$-cycles and one $4$-cycle) and every arc in $D_8$ belongs to exactly
		two cycles in $T^*$. Therefore $\fas(D_8) \geq 13/2$, which implies that $\fas(D_8) \geq 7$, as $\fas(D_8)$ is an integer.
		Define $D_{24}$ as follows.  Take three copies of $D_8$, denoted by $D_8^1$, $D_8^2$ and $D_8^3$, respectively. Let $V(D_8^i)=\{u_1^i,u_2^i,\ldots, u_8^i\}$, such that $u_j^i$ is the copy of $u_j$ in $D_8$. Now add the following four $3$-cycles to $D_8^1 \cup D_8^2 \cup D_8^3$,
		\[
		{\cal C}_4 = \{u_3^1 u_3^2 u_3^3 u_3^1, u_4^1 u_4^2 u_4^3 u_4^1, u_7^1 u_7^2 u_7^3 u_7^1, u_8^1 u_8^2 u_8^3 u_8^1 \}.
		\]
		Let the resulting digraph be denoted by $D$. It is not difficult to see that $D$ is a degree-$6$ oriented graph. Furthermore any feedback arc set of $D$ contains at least one arc from each of the four $3$-cycles in ${\cal C}_4$. So $\fas(D) \geq 4 + 3 \cdot \fas(D_8) \geq 25$. Now we note that the following holds: $\fas(D) \geq 25 = 25 \cdot m / (3 \cdot 24) = \frac{25}{72} m$. Hence, $c'_{\leq 6} \geq 25/72$. Also, $\fas(D)\ge \frac{25}{72} m=\frac{75}{72}n$ and so \AY{$c''_{\leq 6} = c''_6  \geq 75/72$, by Proposition~\ref{relations}(i).}
	\end{proof}


\begin{thebibliography}{99}
		\bibitem{Alon2002} N. Alon, Voting paradoxes and digraphs realizations, Adv. Applied Math. 29(1) (2002) 126--135.
		
		\bibitem{Alon2006} N. Alon, Ranking tournaments, SIAM J. Discrete Math. 20 (2006) 137--142.
		
		\bibitem{BJG} J. Bang-Jensen and G.Z. Gutin. Basic terminology, notation and results. In: J. Bang-Jensen and G.Z. Gutin, editors, Classes of Directed Graphs, Springer Monographs in Mathematics, Springer, 2018.
		
		\bibitem{Berger1997} B. Berger, The fourth moment method, SIAM J. Comput. 26 (1997) 1188--1207.
		
		\bibitem{BS1990} B. Berger and P.W. Shor, Approximation algorithms for the maximum acyclic subgraph problem. In: Proc. First ACM-SIAM Symposium on Discrete Algorithms,
		(1990) 236--243.
		
		\bibitem{BS1997}  B. Berger and P. Shor, Tight bounds for the maximum acyclic subgraph problem, J. Algorithms 3 (1997) 1--18.
		
		
		\bibitem{CharbitTY2007}	P. Charbit, S. Thomasse, and A. Yeo, The minimum feedback arc set problem is NP-hard for tournaments, Combin. Probab. Comput. 16 (2007) 1--4.
		
		
		\bibitem{ELS1993} P. Eades, X. Lin, and \LH{W.F.} Smyth, A fast and effective heuristic for the feedback arc set problem, Inform. Process. Lett. 47 (1993) 319--323.
		
		\bibitem{Hanauer2017} K. Hanauer, Linear Orderings of Sparse Graphs, PhD thesis, Universit{\"a}t Passau, 2017.
		
		\bibitem{HBA2013} K. Hanauer, F.J. Brandenburg and C. Auer (2013). Tight Upper Bounds for Minimum Feedback Arc Sets of Regular Graphs. In: Graph-Theoretic Concepts in Computer Science (WG 2013), Lect. Notes Comput. Sci., 8165 (2013) 298--309, Springer, Berlin.
		
		
		\bibitem{Jung1970} H.A. Jung, On subgraphs without cycles in tournaments. In: Combinatorial Theory \& Its Applications II (North-Holland, Amsterdam,
		1970) 675--677.
		
		\bibitem{LS1991} C.E. Leiserson and J.B. Saxe, Retiming synchronous circuitry, Algorithmica 6 (1991) 5--35.
		
		\bibitem{Spencer1971} J. Spencer, Optimal ranking of tournaments, Networks 1 (1971) 135-138.
		
		\bibitem{Spencer1980} J. Spencer, Optimally ranking unrankable tournaments, Period. Math. Hungar. 11(2) (1980) 131--144.
		
	\end{thebibliography}
\end{document}